\documentclass[10pt,a4paper]{amsart}
\usepackage[utf8]{inputenc}
\usepackage[T1]{fontenc}
\usepackage{amsmath}
\usepackage{amssymb}
\usepackage{graphicx}
\usepackage{hyperref}
\usepackage{cleveref}
\usepackage{tikz-cd}
\usepackage[margin=3.5cm]{geometry}

\newtheorem{thm}{Theorem}[section]
\newtheorem{cor}[thm]{Corollary}
\newtheorem{lem}[thm]{Lemma}
\newtheorem{prop}[thm]{Proposition}
\theoremstyle{definition}
\newtheorem{rem}[thm]{Remark}

\newtheorem{thmintro}{Theorem}

\title{Formality is preserved under domination}

\begin{document}
    \author{A. Milivojevi\'c}
	\address{Max-Planck-Institut f\"ur Mathematik,
		Vivatsgasse 7,
		53111 Bonn}
	\email{milivojevic@mpim-bonn.mpg.de}
	\author{J. Stelzig}
	\address{ Mathematisches Institut der Ludwig-Maximilians-Universit\"at M\"unchen,
		Theresienstraße 39, 80993 M\"unchen}
	\email{jonas.stelzig@math.lmu.de}
	\author{L. Zoller}
	\address{ Mathematisches Institut der Ludwig-Maximilians-Universit\"at M\"unchen,
		Theresienstraße 39, 80993 M\"unchen}
	\email{zoller@math.lmu.de}
	
	\begin{abstract}
If a closed orientable manifold (resp.\ rational Poincar\'e duality space) $X$ receives a map $Y \to X$ from a formal manifold (resp. space) $Y$ that hits a fundamental class, then $X$ is formal. The main technical ingredient in the proof states that given a map of $A_\infty$-algebras $A\to B$ admitting a homotopy $A$-bimodule retract, formality of $B$ implies that of $A$.
	\end{abstract}

\maketitle

\setcounter{tocdepth}{1}
\tableofcontents

\section{Introduction}

A basic relation one can consider among manifolds is that of \emph{domination}: For equidimensional closed orientable manifolds, one says $Y$ dominates $X$ if there is a non-zero degree map $Y \to X$. In this situation, a general heuristic says that $X$ is ``simpler'' than $Y$. From a rational homotopy theoretic point of view, a formal space is the simplest space with a given cohomology ring and it is therefore natural to ask whether the property of formality is preserved by dominant maps. Formality here refers to the property that the commutative differential graded algebra of differential forms can be connected by quasi-isomorphisms to its cohomology equipped with trivial differential. In line with the above heuristic, our main result is:

\begin{thmintro}\label{thmA}
If $Y$ dominates $X$, and $Y$ is formal, then $X$ is formal.
\end{thmintro}

In fact, we prove the result for the following two slightly different generalizations of the notion of dominance of a map $f:Y\to X$, without assumptions on the dimensions, both in the spirit of \cite{CT89}.

\begin{enumerate}

\item\label{sit:PD} $f$ is a continuous map from a space to a rational Poincar\'e duality space, inducing a surjection in top degree rational homology of $X$.\footnote{Here, formality refers to that of the cdga of piecewise polynomial rational forms in the sense of Sullivan \cite{S77}.}

\item\label{sit:mfds} $f$ is a proper, smooth map between smooth orientable manifolds, such that a fundamental class in Borel-Moore homology of $X$ is in the image of $f_*$.
\end{enumerate}
It is equivalent to require surjectivity in (Borel--Moore) homology in all degrees. Both cases overlap in the case of $f:Y\to X$ being a map of smooth closed orientable manifolds.

For example, $f$ could be a finite (ramified) covering map or an orientable fibration with surjective restriction map $H(Y)\to H(F)$ to the cohomology of the fibre. Applying this to the twistor fibration of a compact positive quaternion-K\"ahler manifold, one recovers formality of the latter, first proved in \cite{AK12}. Likewise, $X$ could be an algebraic variety satisfying rational Poincar\'e duality and $f$ a resolution of singularities, recovering \cite[Theorem 5]{H86}, see also \cite[Section 3]{ChCi17}. Furthermore, since two-dimensional surfaces are formal, the fundamental class of any non-formal oriented manifold cannot be mapped to by a product of surfaces (or any other formal manifolds), confirming a remark of Gromov in these cases \cite[p.301]{G99}; this complements results by Kotschick--Löh \cite{KL09}, who exhibited obstructions to domination by products of a different nature. 

Our results are inspired by \cite[Theorem 5.22]{DGMS75}, see also \cite{Me22}, that the $\partial \overline{\partial}$-lemma is preserved under dominant holomorphic maps of compact complex manifolds, and a theorem of Taylor \cite{Ta10} that non-trivial triple Massey products pull back non-trivially under non-zero degree maps of rational Poincar\'e spaces. In the case of $X$ and $Y$ being rational Poincar\'e duality spaces of dimension $\leq 5n+2$ where $X$ is cohomologically $n$--connected, \Cref{thmA} follows from the naturality of the Bianchi--Massey tensor constructed in \cite{CN20}.

As the essential argument and computation is contained in the case of $Y$ and $X$ being closed manifolds of the same dimension $n$, let us outline how to treat this case: The map $f$ gives rise to a commutative diagram
    \begin{equation}\label{diag: duality}
\begin{tikzcd}
DA_X[n]&&DA_Y[n]\ar[ll,"\frac{1}{d}(f^*)^\vee",swap]\\
A_X\ar[u,"\Phi_X"]\ar[rr, "f^*"]&&A_Y\ar[u,"\Phi_Y",swap],
\end{tikzcd}
\end{equation}
where $A_X,A_Y$ denote the cdga's of differential forms and $DA_X[n]$, $DA_Y[n]$ are the (degree shifted) dual complexes of $A_X$, resp. $A_Y$, which are naturally differential graded modules over $A_X$, resp.\ $A_Y$ by precomposition. The vertical maps are the morphisms of differential graded modules given by wedge-product and integration over a fundamental class. By Poincaré duality they are in fact quasi-isomorphisms.

We would like to invert $\Phi_X$ to obtain a module-retract of $f^*$, at least up to homotopy. This is not always possible in the world of cdga's and their modules. It is however, if we work with $A_\infty$-algebras and their bimodules, recalled below. These also have the advantage of making the obstructions to formality transparent in terms of higher operations (operadic versions of the classical Massey products). \Cref{thmA} then follows from the following purely algebraic statement:

\begin{thmintro}\label{thmB}
		Let $f:A\to B$ be a morphism of $A_\infty$-algebras in characteristic $0$ admitting a $A_\infty$-$A$-bimodule homotopy retract. If $B$ is formal, then so is $A$. More generally, if $B$ is quasi-isomorphic to a minimal $A_\infty$-algebra with $m_i = 0$ for $3\leq i \leq k$ for some $k\in\mathbb N\cup\{\infty\}$, then the same holds for $A$.

\end{thmintro}

As another application, we can recover descent of formality in characteristic zero, namely that formality after field extension implies formality, see \Cref{descent}. We note that \Cref{thmB} is a generalization of the well-known fact that a retract of a formal space is formal \cite[Example 2.88]{FOT08}, where the condition of the retract being, on the algebraic level, a morphism of ($A_\infty$-) algebras has been weakened to only being a morphism of $A_\infty$-$A$-bimodules. 

 The structure of the article is as follows: In Section 2 we concisely recall the required algebraic machinery in the form of $A_\infty$-algebras and bimodules. Subsequently, Theorem \ref{thmB} is proved in Section 3. The final Section 4 then has the purpose of transferring the argument surrounding Diagram \ref{diag: duality} outlined above to the more general geometric setups in Theorem \ref{thmA} and proving the latter.

\subsection*{Acknowledgements} A.M. thanks the MPIM in Bonn for its support, together with the LMU Munich, where part of this work was carried out, for their generous hospitality; the authors thank Dieter Kotschick for his comments.

\section{Generalities on algebras and modules} 

	We will always work over a field of characteristic zero. For a graded object $A$, we denote its suspension by $sA$. It has the same underlying space, with grading $(sA)^k=A^{k+1}$. 

	\subsection{$A_\infty$-algebras and bimodules}
    
    For a vector space $V$ we denote by $TV=\bigoplus_{i\geq 0} V^{\otimes i}$ the tensor coalgebra over $V$. Summing only over $i\geq 1$ we obtain the reduced tensor coalgebra $\overline{T}V$.
	An $A_\infty$-algebra structure on a vector space $A$ is a degree $1$ coderivation $D$ on $\overline{T}sA$ which squares to zero. Equivalently, a coderivation on $\overline{T}sA$ can be specified by a collection of maps $d_k:sA^{\otimes k}\to sA$, $k\geq 1$, of degree $1$, which can always be extended uniquely to a coderivation $D\colon \overline{T}sA\rightarrow \overline{T}sA$. The condition $D^2=0$ is equivalent to
	\begin{equation}\label{eqn: Ainfty}
		\sum_{a+b+c=n}d_{a+c+1}(1^{\otimes a}\otimes d_b \otimes 1^{\otimes c})=0
	\end{equation}
	for $n\geq 1$; here and throughout, the indices of summation are understood to be non-negative, and furthermore terms with invalid indices, e.g.\ $d_0$ (or $f_0$, $r_{a,{-1}})$ below) are set to zero. Similarly, given an $A_\infty$-algebra $(A,D)$, an $A_\infty$-bimodule structure over $(A,D)$ on a vector space $M$ is a degree $1$ codifferential $D^M$
    on the $TsA$-cobimodule $TsA\otimes sM\otimes TsA$ such that $D^M\circ D^M=0$; here $TsA=\overline{T}sA\oplus \langle 1\rangle$ inherits its differential from $\overline{T}sA$ by setting it to be trivial on $1$. This is equivalent to a collection of degree $1$ maps $d_{p,q} \colon sA^{\otimes p}\otimes sM\otimes sA^{\otimes q}\to sM$ such that for all $p,q\geq 0$ one has
	\begin{equation}\label{eqn: Ainfty bim}
		\sum_{a+b+c=p+q+1}d_{a,c}(1^{\otimes a}\otimes d_{b}\otimes 1^{\otimes c})=0,
	\end{equation}
	where this is an equation of maps $sA^{\otimes p}\otimes sM\otimes sA^{\otimes q}\to sN$, and $d$ on the left hand side is interpreted either in the module or the algebra sense, depending on its position (i.e. as $d^M_{p-a,q-c}$ or $d_b^A$, respectively), and similarly for the identity operator $1$.
	An $A_\infty$-algebra is a bimodule over itself by setting $d_{p,q}:=d_{p+q+1}$.

	\subsection{Morphisms} A morphism of $A_\infty$-algebras $(A,D^A), (B,D^B)$ is given by a morphism $f:\overline{T}sA\to \overline{T}sB$ of coalgebras such that $fD^A=D^Bf$, or equivalently by a sequence of degree $0$ maps $f_k \colon sA^{\otimes k}\to sB$ such that for every $n\geq 1$,
	\begin{equation}\label{eqn: Algebra morphism}
		\sum_{a+b+c=n} f_{a+c+1}(1^{a}\otimes d_b^A\otimes 1^{\otimes c})=\sum_{i_1+ \cdots +i_r=n} d_r^B(f_{i_1}\otimes \cdots \otimes f_{i_r})
	\end{equation}
	
	Analogously, a map of $A_\infty$-bimodules $M\to N$ over some $A_\infty$-algebra $A$ is given by a morphism between the cobimodules $TsA\otimes sM\otimes TsA\to TsA\otimes sN\otimes TsA$ which commutes with the codifferentials. Again this is described by a collection of  degree $0$ maps $r_{p,q}:sA^{\otimes p}\otimes sM\otimes sA^{\otimes q}\to sN$ such that for every $p,q\geq 0$,
		\begin{equation}\label{eqn: Bimodule morphism}
		\sum_{a+b+c=p+q+1} r_{a,c}(1^{\otimes a}\otimes d_{b}\otimes 1^{\otimes c})=\sum_{\substack{x+i=p\\y+j=q}} d_{x,y}^N(1^{\otimes x}\otimes r_{i,j}\otimes 1^{\otimes y}),
	\end{equation}
	where again this is an equation of maps from $sA^{\otimes p}\otimes sM\otimes sA^{\otimes q}\to sN$, and $d$ on the left hand side is interpreted either in the module or the algebra sense, and similarly for the identity operator $1$.

 \subsection{Signs, suspensions, and the classical notions} One can rewrite all the above equations without suspensions, at the expense of introducing signs according to the Koszul sign rule. In this case, the notation $m_k$ for the maps $s^{-1}\circ d_k\circ (s^{\otimes k})\colon A^{\otimes k}\rightarrow A$ on the unshifted spaces is more common, where $s\colon A\rightarrow sA$ is the suspension map of degree $-1$. Equation \ref{eqn: Ainfty} then becomes
 \[\sum_{a+b+c=n} (-1)^{ab+c}m_{a+c+1}(1^{\otimes a}\otimes m_b \otimes 1^{\otimes c})=0. \]
 From this one verifies that the structure of a non-unital, associative differential graded algebra (dga) on a vector space $A$ is the same as an $A_\infty$-structure on $A$ with $m_i=0$ for $i\geq 3$. Indeed $m_1$ takes the role of the differential, while $m_2$ is the multiplication. Similar considerations for morphisms show that the category of dgas embeds into that of $A_\infty$-algebras. The analogous statement holds for the category of dg $A$-bimodules over a dga $A$. Throughout the paper all algebraic structures are viewed in their respective non-unital categories unless stated otherwise.
 
 To minimize sign calculations, we work with the maps $d_k$ instead of the $m_k$ throughout.
	
	\subsection{Restriction of scalars}
	If $f:A\to B$ is a map of $A_\infty$-algebras and $M$ a $B$-bimodule, it inherits a structure of an $A$-bimodule by defining 
	\begin{equation}\label{eqn: restriction of scalars}
		d^{M/A}_{p,q}:=\sum_{\substack{i_1+\cdots+i_r=p\\
				j_1+\cdots+j_s= q}}d^M_{r,s}(f_{i_1}\otimes \cdots \otimes f_{i_r}\otimes 1_M\otimes f_{j_1}\otimes \cdots \otimes f_{j_s}).
	\end{equation}	
	In particular, one can apply this formula to $B$ itself and then $f$ induces also a map of $A$-bimodules by setting $f_{a,b}=f_{a+b+1}$. 
	For fixed $f$, restriction defines a functor $R_f$
	from $B$-bimodules to $A$-bimodules by sending a map $r:M\to N$ of $B$-bimodules to a map $R_f(r)$ of the induced $A$-bimodules defined via
	\begin{equation}
		R_f(r)_{p,q}:=\sum_{\substack{i_1+\cdots+i_a=p\\
				j_1+\cdots+j_b= q}} r_{a,b}(f_{i_1}\otimes \cdots \otimes f_{i_a}\otimes 1_M\otimes f_{j_1}\otimes \cdots \otimes f_{j_b})
	\end{equation}
	Given two maps of $A_\infty$-algebras $A\overset{f}{\to} B\overset{g}{\to} C$, restriction is compatible in the sense that $R_{g\circ f}=R_f\circ R_g$.	

\subsection{Minimality, quasi-isomorphisms, and formality} \Cref{eqn: Ainfty} in arity $1$ implies that $d_1^2=0$, so $(sA,d_1)$ is a complex and one can consider its cohomology. 
A map of $A_\infty$-algebras $f\colon A\to B$ is called a quasi-isomorphism if the induced map $f_1\colon (sA,d_1)\to (sB,d_1)$ is a quasi-isomorphism.

Similarly a morphism $M\rightarrow N$ of $A_\infty$-bimodules is a quasi-isomorphism if $r_{0,0}\colon (sM,d_{0,0})\rightarrow (sN,d_{0,0})$ is.
Any quasi-isomorphism of $A_\infty$-algebras (resp.\ bimodules) admits a quasi-inverse, i.e.\ a quasi-isomorphism in the opposite direction, which induces the inverse map in cohomology \cite[p.94]{LH03} \footnote{This applies in the augmented setting, which we can appeal to by \cite[p.81]{LH03}.}. An $A_\infty$-algebra is called minimal, if $d_1=0$. Any $A_\infty$-algebra is quasi-isomorphic to a minimal one.

An $A_\infty$-algebra $A$ is called formal if it is quasi-isomorphic to its own cohomology, i.e.\ there is a minimal $A_\infty$-model $H(A)\rightarrow A$ where the $d_i$ on the left hand side vanish, except for $d_2$, which is the natural product induced on cohomology thanks to \cref{eqn: Ainfty} in arity $2$.
\begin{rem}\label{rem: formality different cats}
    When talking about formality of spaces one traditionally means the formality of an associated algebra of forms (deRham or piecewise linear) in the category of unital cdgas. However a unital cdga is formal in its own category (i.e.\ quasi-isomorphic to its cohomology through unital cdgas) if and only if it is formal when considered as an $A_\infty$-algebra. Indeed, any quasi-isomorphism (in the category of non-unital $A_\infty$-algebras) between $A$ and $H(A)$ will respect the cohomological units, and is hence homotopic to a strictly unital map \cite[Cor. 3.2.4.4]{LH03}. Thus, by \cite[Theorem 3.3., 3.17 (3)]{CPRNW19}, \cite{S17}, $A$ is formal as a unital cdga.
\end{rem}

\section{Proof of \Cref{thmB}}

	\subsection{Setup}
	A map of $A_\infty$-algebras $f:A\to B$ is said to admit a ($A_\infty$-$A$-bimodule) homotopy retract if there exists a map $r:B\to A$ of $A_\infty$-$A$-bimodules such that $H(r_{0,0}\circ f_1)=\mathrm{Id}$. \footnote{We note that this terminology is justified since the $A_\infty$ $A$-bimodule automorphism of $A$ resulting from $r$ and $f$ will be a quasi-isomorphism and therefore a homotopy equivalence.}
	
	\begin{lem}\label{lem: minimal retract replacement}
	If $f:A\to B$ admits a homotopy retract and $A'\overset{p}{\to} A$, $B\overset{q}{\to} B'$ are quasi-isomorphisms of $A_\infty$-algebras, then the induced map $q\circ f\circ p: A'\to B'$ admits a homotopy retract.
	\end{lem}

\begin{proof} The algebra map $p$ induces a quasi-isomorphism $\tilde{p}\colon A'\rightarrow A$ of $A'$-bimodules. Similarly we obtain a $B$-bimodule quasi-isomorphism $B\rightarrow B'$ which we restrict to a map $\tilde{q}$ of $A'$-bimodules along $f\circ p$. There is an $A$-bimodule map $r$ as above, which induces a morphism $\tilde{r}$ of the restricted $A'$-bimodules. Now for $A'$-bimodule quasi-inverses $\tilde{p}',\tilde{q}'$ of $\tilde{p},\tilde{q}$ the composition $\tilde{p}'\circ \tilde{r}\circ \tilde{q}'$ gives the desired homotopy retract.\end{proof}

We will prove \Cref{thmB} for $k\in\mathbb{N}$ via an induction over $k$. Using \Cref{lem: minimal retract replacement} we may replace $A,B$ by arbitrary models in their quasi-isomorphism types. For some fixed $k\geq 2$ we assume that $B$ is minimal and $D^B$ is such that $d_i^B=0$ for $3\leq i\leq k+1$. We also assume $A$ to be minimal and that by induction $d_i^A=0$ for $3\leq i\leq k$. We note that since $A$ and $B$ are minimal the homotopy retract condition becomes $r_{0,0}\circ f_1= \mathrm{Id}_A$. Our aim is to define a coalgebra automorphism $\varphi\colon \overline{T}sA\rightarrow \overline{T}sA$ such that the transformed differential $\tilde{D}:=\varphi D\varphi^{-1}$ has vanishing components $\tilde{d}_i$ for $3\leq i\leq k+1$. In this setup $\varphi\colon (\overline{T}sA,D)\rightarrow (\overline{T}sA,\tilde{D})$ will be an isomorphism of $A_\infty$-algebras and hence finish the induction.

To treat the case $k=\infty$, one can then compose all of these automorphisms. It will be clear from their explicit form that this composition involves only a finite number of terms in every arity, so the infinite composition gives a morphism of $A_\infty$-algebras.

\begin{rem}
 The above induction can be made sense of from an obstruction theoretic point of view akin to \cite{HS79}: one can view the formality obstruction $d_{k+1}\colon A^{\otimes (k+1)}\rightarrow A$ as a cocycle in the complex of coderivations on $\overline{T}sA$. Our automorphism $\varphi$ will be defined via a map $\varphi_k\colon A^{\otimes k}\rightarrow A$, which can be interpreted as a coderivation making the above cocycle exact.
\end{rem}

	\subsection{Ansatz}
	Our Ansatz will be to define the components of $\varphi$ as $\varphi_j=0$ unless $j=1,k$, with $\varphi_1=\mathrm{Id}$ and $\varphi_k$ to be determined. The transformed differential $\tilde{D}:=\varphi D\varphi^{-1}$ satisfies $\varphi \circ D= \tilde D\circ \varphi$. Breaking down this equation according to its components, we obtain $\tilde{d}_j=d_j$ for $j\leq k$ and
	\begin{equation}
		\tilde{d}_{k+1}+d_2(\varphi_k\otimes 1+1\otimes \varphi_k)= d_{k+1}+\varphi_k\left(\sum_{p+q=k-1}1^{\otimes p}\otimes d_2\otimes 1^{\otimes q}\right).
	\end{equation}
	Thus, in order to achieve $\tilde{d}_{k+1}=0$, we need to choose $\varphi_k$ so that 
	\begin{equation}
	d_{k+1}=d_2(\varphi_k\otimes 1+1\otimes \varphi_k)-\varphi_k\left(\sum_{p+q=k-1}1^{\otimes p}\otimes d_2\otimes 1^{\otimes q}\right).
	\end{equation}
	Abusing notation slightly, we will write $[d_2,\varphi_k]$ for the right hand side of this equation. We note that $[d_2,\varphi_k]$ is a linear expression in $\varphi_k$.
	We define 
	\begin{align}
		\varphi_{k}^{(i)}&:= \sum_{a+b=i}r_{a,b}(1^{\otimes a}\otimes f_{k-i}\otimes 1^{\otimes b}),\\
		\varphi_{k}&:=\sum_{i=0}^{k-2} \left(1-\frac {i}{k-1}\right)\cdot \varphi_k^{(i)}.
	\end{align}

	\subsection{Preliminary calculations} We now do the necessary calculations to verify our Ansatz works, i.e.\ that $d_{k+1} = [d_2,\varphi_k]$. Schematically, we use the morphism equations to move the instances of $d$ between those of $r$ and $f$, and look for cancellation of terms.
	
	\begin{lem}[Morphism equations for $r$]\label{lem: morphism equations for r}
		For $p+q\leq k-1$, the morphism  \cref{eqn: Bimodule morphism} for $r$ on $sA^{\otimes p} \otimes sB\otimes sA^{\otimes q}$ reads
		\begin{align*}
			d_2^A(r_{p,q-1}\otimes 1_A+1_A\otimes r_{p-1,q})=&\phantom{+}\sum_{a<p}r_{a,q}(1_A^{\otimes a}\otimes d_2^B(f_{p-a}\otimes 1_B)\otimes 1_A^{\otimes q})\\ &+ \sum_{b<q}r_{p,b}(1_A^{\otimes p}\otimes d_2^B(1_B\otimes f_{q-b})\otimes 1_A^{\otimes b})\\ &+ \sum_{x+y=p-2} r_{p-1,q}(1_A^{\otimes x}\otimes d_2^A\otimes 1_A^{\otimes y}\otimes 1_B\otimes 1_A^{\otimes q})\\ &+ \sum_{x+y=q-2}r_{p,q-1}(1_A^{\otimes p}\otimes 1_B\otimes 1_A^{\otimes x}\otimes d_2^A\otimes 1_A^{\otimes y}).
		\end{align*}
	\end{lem}
\begin{proof}
	This follows from \cref{eqn: Bimodule morphism} and \cref{eqn: restriction of scalars}, where we use that in our range, all $d_j^A=d_j^B=0$ unless $j=2$. In particular, $d_{r,s}^{B/A}=0$ unless $r=0$ or $s=0$. In those cases, we have $d_{r,0}^{B/A}=d_2^B(f_r\otimes 1_B)$ and $d_{0,s}^{B/A}=d_2^B(1_B\otimes f_s)$.
\end{proof}
	
	\begin{lem}[The first half of the commutator]\label{lem: first half of commutator}
		For $i\leq k-2$, there is an equality 
		\begin{align*}
			d_2(\varphi_k^{(i)}\otimes 1_A+1_A\otimes \varphi_k^{(i)})=&\phantom{+}\sum_{l\leq i}\sum_{a+b=l}r_{a,b}(1_A^{\otimes a}\otimes d_2(f_{k-i}\otimes f_{1+i-l}+f_{1+i-l}\otimes f_{k-i})\otimes 1_A^{\otimes b})\\
			&+\sum_{a+b=i}\sum_{x+y=a-1} r_{a,b}(1_A^{\otimes x}\otimes d_2\otimes 1_A^{\otimes y}\otimes f_{k-i}\otimes 1_A^{\otimes b})\\
			&+\sum_{a+b=i}\sum_{x+y=b-1} r_{a,b}(1_A^{\otimes a}\otimes f_{k-i}\otimes 1_A^{\otimes x}\otimes d_2\otimes 1_A^{\otimes y}).
		\end{align*}
	\end{lem}
\begin{proof}
	This follows by plugging $f_{k-i}$ into the equation in \Cref{lem: morphism equations for r} when $p+q=i+1$ and then summing over all these pairs $(p,q)$.
\end{proof}

\begin{lem}[The second half of the commutator]\label{lem: second half of commutator}
	For any $0\leq i\leq k-2$, there is an equality
	\begin{align*}
		\varphi_k^{(i)}\left(\sum_{p+q=k-1}1^{\otimes p}\otimes d_2\otimes 1^{\otimes q}\right)=&\phantom{+}\sum_{a+b=i}\sum_{x+y=a-1}r_{a,b}(1^{\otimes x}_A\otimes d_2\otimes 1_A^{\otimes y}\otimes f_{k-i}\otimes 1_A^{\otimes b})\\
		&+\sum_{a+b=i}\sum_{r+s=k+1-i}r_{a,b}(1^{\otimes a}_A\otimes d_2(f_r\otimes f_s)\otimes 1_A^{\otimes b})\\
		&+\sum_{a+b=i}\sum_{x+y=b-1}r_{a,b}(1_A^{\otimes a}\otimes f_{k-i}\otimes 1_A^{\otimes x}\otimes d_2\otimes 1_A^{\otimes y})\\
		&-\delta_{0}^i\cdot d_{k+1},
	\end{align*}
where $\delta_{0}^i=0$ for $i\neq 0$ and $\delta^{0}_0=1$.
\end{lem}

\begin{proof}
The formula in the statement follows by applying the definition to the left hand side and plugging in the following simplified version of \cref{eqn: Algebra morphism} for $k+1-i$ inputs (where we have only used that most components of $d$ vanish):
\begin{align*}
	\sum_{x+y=k-i-1}f_{k-i}(1_A^{\otimes x}\otimes d_2\otimes 1_A^{\otimes y}) + \delta_{0}^i\cdot f_1(d_{k+1})=\sum_{r+s=k+1-i} d_2(f_r\otimes f_s).
\end{align*}
Note that to conclude we use $r_{00}\circ f_{1}=\mathrm{Id}$.
\end{proof}
By subtracting the two sides of the commutator, we obtain:
\begin{cor}[The full commutator on each component]\label{cor: full commutator}
	For any $0\leq i\leq k-2$, there is an equality
	\begin{align*}
		[d_2,\varphi_k^{(i)}]=&\phantom{+}\sum_{l<i}\sum_{a+b=l} r_{a,b}(1_A^{\otimes a}\otimes d_2(f_{k-i}\otimes f_{1+i-l}+f_{1+i-l}\otimes f_{k-i})\otimes 1_A^{\otimes b})\\
		&-\sum_{a+b=i}\sum_{\substack{\, r+s=k+1-i\\r,s\geq 2}} r_{a,b}(1_A^{\otimes a}\otimes d_2(f_r\otimes f_s)\otimes 1_A^{\otimes b})\\
		&+\delta_{0}^i\cdot d_{k+1}.
	\end{align*}
\end{cor}

\subsection{Computing the commutator.}
Recall that we need to show that $[d_2,\varphi_k]=d_{k+1}$. By the formula in \Cref{cor: full commutator}, we obtain
\begin{align*}
	[d_2,\varphi_k]=&\sum_{i=0}^{k-2}\left(1-\frac{i}{k-1}\right)\cdot [d_2,\varphi_k^{(i)}]\\
	& =d_{k+1} + \sum_{\substack{a+b+r+s=k+1\\a,b\geq 0\\r,s\geq 2}} C_{a,b,r,s}\cdot r_{a,b}(1^{\otimes a}\otimes d_2(f_r\otimes f_s)\otimes 1^{\otimes b})
\end{align*}
for some coefficients $C_{a,b,r,s}\in \mathbb{Q}$. To compute these coefficients, we note that they only receive contributions from the summands corresponding to $i=a+b, a+b+r-1, a+b+s-1$. 

For the moment, let us assume $r\neq s$, so that these are really three distinct summands. Then, using that $a+b+r+s=k+1$, we have
\[
C_{a,b,r,s}=-\left(1-\frac{a+b}{k-1}\right)+\left(1-\frac{a+b+r-1}{k-1}\right)+\left(1-\frac{a+b+s-1}{k-1}\right)=0.
\]
In the case that $r=s$, the summand $r_{a,b}(1_A^{\otimes a}\otimes d_2(f_r\otimes f_r)\otimes 1_A^{\otimes b})$ appears twice in $[d_2,\varphi_k^{(a+b+r-1)}]$, so the same calculation remains valid.


\begin{rem}\label{descent} From \Cref{thmB} we recover descent of formality in characteristic zero, see \cite[Theorem 12.1]{S77}. Namely, let $k \subset K$ be an extension of fields of characteristic zero, and let $A$ be a $k$-cdga such that $A\otimes_k K$ is formal as a $K$-cdga. Then $A \otimes_k K$ is also formal as a $k$-cdga. Picking a $k$-vector space retract for the inclusion $k \subset K$ induces an $A$-module retract of $A \hookrightarrow A \otimes_k K$, so we are in the setting of \Cref{thmB}. \end{rem}

\section{Proof of \Cref{thmA}}

It remains to show how to arrive at the setting of \Cref{thmB} from Situations  \ref{sit:PD} and \ref{sit:mfds}. 
\subsection{Algebraic preliminaries} Let $A$ be a cdga and $M$ a dg-module over $A$. Since $A$ is commutative, we will not distinguish between left-, right-, and bimodules. For instance, a right dg module structure induces a dg bimodule structure via $a.m:=(-1)^{|m||a|} m.a$. 

For any closed element $m\in M^0$, we obtain a degree $0$ map of dg modules $A\to M$ via $a\mapsto a.m$. For any $k$, the $k$-shifted dual space $D_kM$, with grading $(D_kM)^l:=(M^{k-l})^\vee$, is again a dg module via setting $(d\varphi)(m):=(-1)^{|\varphi|+1}\varphi(dm)$ and $(\varphi.a)(m):=\varphi(a.m)$. In particular, for any closed element $\phi\in (D_kM)^0=(M^k)^\vee$, we obtain a map of dg modules $A\to D_kM$ as above, which we denote by $m_\phi$. The induced map in cohomology depends only on the class $[\phi]\in H^0(D_kM)$. If $m_\phi$ is a quasi-isomorphism, we call $[\phi]$ an $M$-dualizing class. For example, a cdga $A$, considered as a module over itself, satisfies Poincar\'e duality on its cohomology if and only if there exists an $A$-dualizing class.\footnote{We note that in case $M=A^\vee$, the map $m_\phi$ is an $\infty$-inner-product in the sense of \cite{T08}, \cite{TZS07}.}

\subsection{Finishing the proof} The following proposition abstracts the algebraic structure underlying Situations \ref{sit:PD} and \ref{sit:mfds}:
\begin{prop}
Let $f:A\to B$ be a map of cdga's and let $M\subseteq A$, $N\subseteq B$ be $A$-submodules such that $f(M)\subseteq N$. Assume that
\begin{enumerate}
    \item There exists an $M$-dualizing class $\mathfrak c\in H^0(D_kM)$.
    \item There exists a class $\mathfrak c'\in H^0(D_kN)$ such that $f_*\mathfrak c'=\mathfrak c$.
\end{enumerate}
Then, if $B$ is formal, so is $A$.
\end{prop}

\begin{proof}
Pick representatives $\mathfrak c=[\phi],\mathfrak c'=[\psi]$ such that $\psi\circ f=\phi$. From the assumptions, we obtain a commutative diagram of dg-$A$-modules:
   \[
    \begin{tikzcd}
    D_kM & D_kN \ar[l,swap, "D_kf"]\\ A\ar[u, "m_\phi"]\ar[r, "f"]& B\ar[u, swap, "m_\psi"]
    \end{tikzcd}
    \]
    Indeed, let $a\in A$, $m\in M$. Then \begin{align*}
         D_kf(m_\psi(f(a)))(m)&=(m_\psi(f(a)))(f(m))=(f(a).\psi) (f(m))\\
         &=\psi(f(am))=\phi(am)=m_\phi(a)(m).
            \end{align*}
By assumption, $m_\phi$ is a quasi-isomorphism. Therefore, considering this as a diagram of $A_\infty$-$A$-bimodules, we can find a quasi-inverse to $m_\phi$ and so $f$ admits a homotopy retract. Then by \Cref{thmB}, $A$ is formal as an $A_\infty$-algebra. This implies formality as a (unital) cdga by \Cref{rem: formality different cats}.\end{proof}

To prove \Cref{thmA} in Situation \ref{sit:PD}, pick $A=A_{PL}(X)$, $M=A$, $B=A_{PL}(Y)$, $N=B$. To prove it in Situation \ref{sit:mfds}, pick $A=A_X$ the differential forms on $X$, $M=A_{X,c}$ the compactly supported differential forms, $B=A_Y$ and $N=A_{Y,c}$ and use Poincar\'e duality in the form \cite[5.12]{GHV72}. Note that $H^{BM}_\bullet(X)\cong {H_c^\bullet}(X)^\vee$, so $(2)$ above is in fact equivalent to the fundamental class in Borel--Moore homology being hit, as stated in the introduction.

\end{document}